\documentclass[10pt]{amsart}

\usepackage[alphabetic]{amsrefs}
\usepackage{amsmath}
\usepackage{amssymb}

\usepackage{verbatim}
\usepackage{enumerate}
\usepackage[all]{xy}

\numberwithin{equation}{section}
\newtheorem{thm}[equation]{Theorem}
\newtheorem*{thm*}{Theorem}

\newtheorem{prop}[equation]{Proposition}
\newtheorem*{remark*}{Remark}

\newtheorem{lemma}[equation]{Lemma}

\theoremstyle{definition}

\theoremstyle{remark}
\newtheorem{remark}[equation]{Remark}

\newcommand{\thmref}[1]{Theo\-rem \ref{#1}}

\newcommand{\lemref}[1]{Lem\-ma \ref{#1}}

\newcommand{\figref}[1]{Fig\-ure \ref{#1}}

\DeclareMathOperator{\Hom}{hom }
\DeclareMathOperator{\End}{End }

\DeclareMathOperator{\Adj}{Adj }

\DeclareMathOperator{\GL}{GL}

\DeclareMathOperator{\Aut}{Aut}
\DeclareMathOperator{\Isom}{Isom}
\DeclareMathOperator{\Pseudo}{\Psi\hspace*{-0.7mm}\Isom}
\DeclareMathOperator{\Stab}{Stab}

\DeclareMathOperator{\rank}{rank }

\newcommand{\M}{\mathbb{M}}

\newcommand{\pseudo}{\Psi\hspace*{-1mm}\Isom}

\renewcommand{\leq}{\leqslant}

\renewcommand{\bar}{\overline}
\renewcommand{\phi}{\varphi}

\begin{document}

\title{Groups acting on tensor products}
\author{Peter A. Brooksbank}
\address{
	Department of Mathematics\\
	Bucknell University\\
	Lewisburg, PA 17837
}
\email{pbrooksb@bucknell.edu}
\author{James B. Wilson}
\address{
	Department of Mathematics\\
	Colorado State University\\
	Fort Collins, CO 80523\\
}
\email{jwilson@math.colostate.edu}
\date{\today}
\keywords{tensor product, bilinear map, autotopism, pseudo-isometry}
\begin{abstract}
Groups preserving a distributive product are encountered often in algebra.  Examples 
include automorphism groups of associative and nonassociative rings, classical groups, 
and automorphism groups of $p$-groups.  
While the great variety of such products precludes any realistic hope of describing
the general structure of the groups that preserve them, it is reasonable to expect
that insight may be gained from an examination of the universal distributive
products: tensor products. We give a detailed description of the groups
preserving tensor products over semisimple and semiprimary rings, and present
effective algorithms to construct generators for these groups.
We also discuss applications of our methods to algorithmic problems for which
all currently known methods require an exponential amount of work.
\end{abstract}

\maketitle


\section{Introduction}
\label{sec:intro}
Many groups can be described  as the set of linear transformations that preserve
a distributive product.  Obviously this is the case for automorphism groups of algebras, both associative 
and nonassociative.  Among the other well known examples are classical groups, which preserve other 
forms of distributive products, namely reflexive forms; cf. \cite{Artin}*{p. 107}.  
In more subtle ways, automorphisms of 
finite $p$-groups preserve a distributive product
that arises from commutation, via the correspondences of Baer~\cite{Baer:class2}, and of Kaloujnine, Lazard, and Mal'cev 
\cite{Warfield:nil}.  
As the estimates for rings \cite{Neretin}, and for $p$-groups \cite{Higman:enum}, suggest,
however, there are simply too many products to have any realistic expectation of understanding 
the structure of all such groups.   

The goal of this paper is to examine 
the groups preserving tensor products.  
As  tensor products are universal products, 
the structure of these groups informs us of overall structure of groups preserving distributive products.
Tensor products over central simple rings have already been studied in 
\cite{LW:invariants}*{Theorems 3.6 \& 4.1}.  Here we consider tensor products over semiprimary rings
(rings $R$ whose Jacobson radical $J(R)$ is nilpotent and whose quotient $R/J(R)$ is semisimple).
This case is surprisingly complicated because the presence of a nontrivial Jacobson radical.  In many ways
the results and essential points are related to automorphisms of rings of 
strictly lower triangular matrices over arbitrary rings, as studied in \citelist{\cite{Lev:aut-nil}\cite{KL:aut-nil}}.

We begin with a general distributive product $\circ:U\times V\to W$ between abelian groups $U$, $V$, and $W$,
also known as a biadditive or bilinear map, or just {\em bimap}.  
For simplicity we assume that bimaps are {\em full} in that 
$W=U\circ V=\langle u\circ v\colon u\in U,v\in V\rangle$.
An {\em autotopism} of 
a bimap $\circ$
is a triple $(f,g;h)$ in
$\Aut(U) \times \Aut(V)\times\Aut(W)$ satisfying
\begin{align}\label{def:auto}
    (\forall & u\in U, \forall v\in V) &  uf\circ gv & = (u\circ v)^h.
 \end{align}
Our notation accommodates the introduction of left and right scalars, so our homomorphisms are evaluated 
on the right (resp. left) for left modules (resp. right modules), and exponentially for bimodules.

We are principally interested in describing
$\Aut(\circ)$, the group of all autotopisms of a bimap $\circ$. 
Our approach relies on the 
{\em ring of adjoints} of $\circ$, defined as
\begin{align}\label{def:adj}
	\Adj(\circ) & = \{ (f,g)\in \End(U)\times\End(V)^{\rm op}: uf\circ v=u\circ gv\}.
\end{align}
This ring was introduced in \cite{Wilson:unique-cent}*{p.~2654} and is characterized as the largest
ring $R$ acting faithfully on $U$ and $V$ for which $\circ\colon U\times V\to W$ factors through 
$U\otimes_{R} V$.
We show that tensor products and adjoint rings are Galois connected (\thmref{thm:Galois}),
and use this connection to prove the following result.

\begin{thm}\label{thm:main1}
The autotopism group of a bimap $\circ\colon U \times V\to W$ embeds in
\begin{align*}
	N(\Adj(\circ)) & =\{(f,g)\in \Aut(U)\times \Aut(V): \Adj(\circ)^{(f,g)}=\Adj(\circ)\},
\end{align*}
with equality precisely when $\circ$ is a tensor product. The latter condition holds if, and only if,
the uniquely induced homomorphism $U\otimes_{\Adj(\circ)} V\to W$ is an isomorphism.
\end{thm}

In several applications, such as constructing automorphism groups of finite $p$-groups, 
the bimaps $\circ$ that arise are endowed with natural symmetry, and one is primarily
interested in special types of autotopisms called {\em pseudo-isometries}; 
see \cite{Wilson:unique-cent}*{p.~2650}.
We therefore consider alternating and symmetric bimaps  $\circ$, and study the group
of all pseudo-isometries of $\circ$, defined as
\begin{align}\label{def:pseudo}
	   \Pseudo(\circ)  & = \{ f \colon (f,g;h) \in \Aut(\circ) \textnormal{ and } f=g \}.
\end{align}
We prove the following result.

\begin{thm}\label{thm:main2}
For an arbitrary nondegenerate alternating or symmetric bimap $\circ\colon V\times V\to W$, the group
$\Pseudo(\circ)$ embeds naturally in 
\begin{align*}
	N^*(\Adj(\circ)) & = \{ f \colon (f,g)\in N(\Adj(\circ))\textnormal{ and } f=g \},
\end{align*}
with equality precisely when $\circ$ is a symmetric or exterior tensor product. The latter
condition holds if, and only if, the map $V\wedge^{\pm}_{\Adj(\circ)}V\to W$ is an isomorphism.
\end{thm}

From our point of view, the crucial aspect of Theorems~\ref{thm:main1} and~\ref{thm:main2} is
that $\Aut(\circ)$ and $\Pseudo(\circ)$ are shown to act on a known associative, unital ring,
which is easy to construct algorithmically~\citelist{\cite{BW:find-isom}*{Section 4}\cite{BW:slope}}.
So much more is known about the structure of rings than of general distributive 
products, and even basic features, such as the Jacobson radical and simple factors of $\Adj(\circ)$,
clarify the structure of $\Aut(\circ)$. Indeed, Theorems~\ref{thm:norm} and \ref{thm:norm-star} 
give detailed structural descriptions of the groups
 $N(\Adj(\circ))$ and $N^*(\Adj(\circ))$ in the case when $\Adj(\circ)$ 
is semiprimary and separable (which holds, for example, whenever $U$ and $V$ are 
finite-dimensional vector spaces over a field).  
These structural details are often sufficient to compute generators for autotopism groups 
efficiently (say in polynomial time), as was observed in \cite{LW:invariants}*{Theorem 1.3} 
for central simple rings.  
\medskip

The paper is organized as follows.  
In Section \ref{sec:fund} we develop the necessary background on bimaps,
culminating with our Galois connection between bimaps on $U \times V$ and subsets 
of $\End(U)\times\End(V)^{\rm op}$ (\thmref{thm:Galois}).  

In Section~\ref{sec:auto} we study the autotopism group of an arbitrary bimap, proving \thmref{thm:main1}, 
and giving a precise structure theorem for $N(\Adj(\circ))$ in the case when $\Adj(\circ)$ is a semiprimary
and separable (\thmref{thm:norm}). 
We also describe an algorithm 
to construct generators for the normaliser of a finite-dimensional matrix algebra, and hence for the autotopism
group of a tensor product.

In Section~\ref{sec:pseudo}, we consider the special case when $\circ$ is symmetric or alternating. 
We prove~\thmref{thm:main2} and provide an analogue of \thmref{thm:norm} for rings with involutions 
(\thmref{thm:norm-star}).

In the concluding section, we expand on the key applications of our results and algorithms to
the problem of computing automorphism groups of finite $p$-groups, and briefly 
discuss ongoing work in this area.


\section{homotopisms, isotopisms and pseudo-isometries of bimaps}
\label{sec:fund}
Our use of rings and modules is standard.
 A {\em bi-additive map} (or just {\em bimap}) is a function 
$\circ\colon U\times V\to W$, where $U,V,W$ are abelian groups, 
satisfying the two-sided distributive law:
\begin{align*}
 (\forall & u_1,u_2\in U, \forall v\in V)	  & (u_1+u_2)\circ v  = u_1\circ v + u_2\circ v \\
(\forall & u\in U,\forall v_1,v_2\in V)	 & u\circ(v_1+v_2)  = u\circ v_1 + u\circ v_2 .
\end{align*}
Recall that our bimaps are {\em full}, in that $W=U\circ V=\langle u\circ v: u\in U,v\in V\rangle$.

Let $\circ\colon U\times V\to W$ and 
$\diamond\colon U'\times V'\to W'$ be bimaps.
A {\em homotopism} from $\circ$ to $\diamond$ is a triple 
$(f\colon U\to U',g\colon V\to V'; h\colon W\to W')$ 
of homomorphisms satisfying
\begin{align}\label{def:homotopism}
	(\forall & u\in U,\forall v\in V) & (u\circ v)^h & = uf\diamond gv.
\end{align} 
Denote by $\Hom(\circ,\diamond)$ the set of all homotopisms from 
$\circ$ to $\diamond$.

\begin{remark}
For a product $\circ:U\times V\to W$ to determine a nonassociative ring requires that $U=V=W$.  This
can result in products that are not full, for example, the multiplication of a nilpotent Lie algebra is never
full.  The restriction to full bimaps can be avoided by considering ``weak'' homotopisms, which are
triples $(f,g;h)$ where $h$ is defined from $U\circ V\to U'\diamond V'$ instead of from $W\to W'$.
\end{remark}

The class of bimaps together with homotopisms forms a category, 
called the {\em homotopism category}. There are
various natural morphisms on classes of bimaps, such as 
adjoint-morphisms \cite{Wilson:div},
so we name the categories after the morphisms rather than the objects.
We are interested primarily in {\em isotopisms}, namely homotopisms 
whose constituent maps 
are all isomorphisms. Define the {\em autotopism group} of a bimap 
$\circ\colon U\times V\to W$ to be
\begin{equation}
\label{eq:aut}
	\Aut(\circ)   =  \Hom(\circ,\circ) \cap(\Aut(U)\times \Aut(V)\times \Aut (W)).
\end{equation}
We denote the elements of $\Aut (\circ)$ as triples $(f,g;h)$, separating $h$ from $f$ and $g$ to 
distinguish between the two natural restrictions of $\Aut (\circ)$: first on 
$U\times V$, and second on $W$. As $\circ$ is full, $h$ is determined by 
$(u\circ v)^h=uf\circ vg$, so
$\Aut(\circ)$ is naturally and faithfully represented on $U\times V$.

\subsection{Factor equivalence}
\label{subsec:factor}
We now fix two abelian groups $U$ and $V$ and consider bimaps on $U\times V$.
For a bimap $\circ\colon U\times V\to X$ and homomorphism $\tau\colon X\to Y$  
we define the bimap $\circ^{\tau}\colon U\times V\to Y$ by 
\begin{align*}
	(\forall & u\in U,\forall v\in V) & u\circ^{\tau}v & = (u\circ v)^{\tau}.
\end{align*}
For bimaps $\circ\colon U\times V\to X$ and $\diamond\colon U\times V\to Y$, we 
say $\diamond$ \emph{factors through} $\circ$, and write $\circ\rightarrow \diamond$, 
if there is a homomorphism 
$\tau\colon U\circ V\to Y$ such that $\circ^\tau=\diamond$ (equivalently,
if $(1_U,1_V;\tau)$ is a homotopism from $\circ$ to $\diamond$).
We say that $\circ$ and $\diamond$ are {\em factor equivalent}, denoted 
$\circ\leftrightarrow \diamond$, if $\circ\to \diamond$ and $\diamond\to \circ$.
Note that $\circ\leftrightarrow \diamond$ forces $X=U\circ V$ and $Y=U\diamond V$ 
to be isomorphic (as our bimaps are full).  It follows that
$\leftrightarrow$ is an equivalence relation on the class of bimaps on $U\times V$,
and $\to$ is a partial order on the bimaps on $U\times V$ relative to
the equivalence $\leftrightarrow$.
\smallskip

Let $U\otimes V$ denote the usual
tensor product of $U$ and $V$ (as abelian groups) and let
$\otimes\colon U\times V\to U\otimes V$ denote the associated bimap.  
The universal property of tensor products asserts that every bimap 
$\circ\colon U\times V\to X$ factors 
uniquely through $\otimes\colon U\times V\to U\otimes V$. Writing
$\hat{\circ}\colon U\otimes V\to X$ for the implied homomorphism, 
so that $\circ=\otimes^{\hat{\circ}}$, we have $u\circ v=(u\otimes v)^{\hat{\circ}}$
for all $u\in U$ and all $v\in V$.  Hence, for a bimap $\diamond$ on $U\times V$, 
\begin{equation}
	\circ \to \diamond \quad \Longleftrightarrow \quad \ker \hat{\circ}\subseteq \ker \hat{\diamond}.
\end{equation} 
\figref{fig:def-hat-b} illustrates this correspondence.

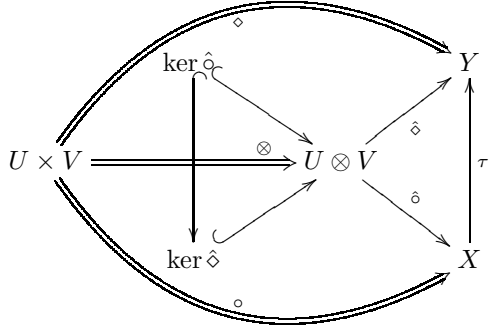
\begin{figure}[!htbp]
\begin{equation*}
\xymatrix{
	& \ker\hat{\circ}~ \ar@{^(->}[dr]\ar@{^(->}[dd] 	& & Y\\
	U\times V\ar@/^40pt/@{=>}[urrr]_{\diamond} \ar@{=>}[rr]^>>>>>{\otimes}\ar@/_40pt/@{=>}[rrrd]^{\circ} & & U\otimes V\ar[ur]_{\hat{\diamond}}\ar[dr]^{\hat{\circ}} &\\
	& \ker\hat{\diamond}\ar@{^(->}[ur]	& & X\ar[uu]_{\tau}
}
\end{equation*}
\caption{Commutative diagram demonstrating how the bimap $\diamond$ factors through $\circ$ if, and only 
if, their factorisations through the tensor product have nested kernels. 
{\em Double shafted arrows denote 
bimaps and regular arrows are homomorphisms.}}\label{fig:def-hat-b}
\end{figure}

Associated to each subgroup $K$ of $U\otimes V$ is a bimap 
\begin{align*}
	\bullet& =\bullet(K)\colon U\times V\to (U\otimes V)/K,
\end{align*}
called the {\em regular bimap mod $K$}, defined by
\begin{align}\label{def:regular}
 (\forall & u\in U,\forall v\in V)
& u\bullet v & =u\otimes v+K.
\end{align}
\figref{fig:def-hat-b} also shows that, for every 
bimap $\circ\colon U\times V\to X$, we have 
$\circ\leftrightarrow \bullet(\ker \hat{\circ})$.  We regard $\bullet(\ker\hat{\circ})$ as the 
\emph{regular representation} of $\circ$ since it is a canonical representative of the factor-equivalence class 
containing $\circ$. 
This establishes a bijection $\circ\mapsto \ker \hat{\circ}$ from the factor-equivalence classes on 
$U\times V$ to the set of subgroups of $U\otimes V$.

We now define meets and joins for bimaps on $U\times V$ in a manner that respects factor equivalence.  
For bimaps $\circ\colon U\times V\to X$ and $\diamond\colon U\times V\to Y$, 
let 
\begin{equation*}
W= \{(u\circ v,u\diamond v)\in X\oplus Y: u\in U, v\in V\}\leq X\times Y, 
\end{equation*}
and
define 
$(\circ\cap \diamond)\colon U\times V\to W$ by
\begin{align}
\label{def:meet}
	(\forall & u\in U, \forall v\in V) & 
	u(\circ\cap \diamond)v & = (u\circ v, u\diamond v).
\end{align}
For the join $\circ\cup \diamond$, let 
$J=\langle(u\circ v, -u\diamond v) \colon u\in U, v\in V\rangle\leq X\times Y$.
Now define $(\circ\cup \diamond)\colon U\times V\to W/J$ so that 
\begin{align}
\label{def:join}
	(\forall &  u\in U,\forall v\in V) & 
	u(\circ\cup \diamond) v & = (u\circ v,u\diamond v) + J.
\end{align}
Both $\cap$ and $\cup$ generalize to arbitrary sets of bimaps on $U\times V$. 

\begin{figure}
\begin{equation*}
\xymatrix{
\\
	&  & U\circ V \ar[dr]\ar@/^/[drr]^{\tau} & \\
U\times V\ar@/^5pc/@{=>}[rrrr]^{\star} \ar@{=>}[r]^{\otimes}\ar@/^/@{=>}[urr]_{\circ}\ar@/_/@{=>}[drr]_{\diamond}  
		& U\otimes V\ar[ur]_{\hat{\circ}}\ar[dr]^{\hat{\diamond}} & & U(\circ\cup \diamond)V \ar[r] & Z. \\
	&  & U\diamond V \ar[ur]\ar@/_/[urr]_{\mu} & 
}
\end{equation*}
\caption{Commutative diagram showing that if $\circ\to\star$ and
$\diamond\to \star$ then $(\circ \cup \diamond)\to \star$. 
The point is that the middle of the diagram is a pushout for $(\hat{\circ}, \hat{\diamond})$.}
\label{fig:join}
\end{figure}
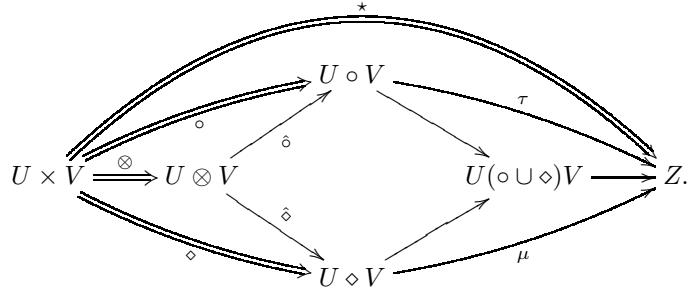

\begin{prop}
\label{prop:lattice}
The factor-equivalence classes of bimaps on $U\times V$ form a lattice under $\to$, $\cap$ and $\cup$  that
is isomorphic to the lattice of subgroups of $U\otimes V$ under $\subseteq$, $\cap$, and $+$.  
A bimap $\circ$ is at the top if 
$U\circ V=0$, and at the bottom if $\hat{\circ}\colon U\otimes V\to U\circ V$ is an isomorphism.  
\end{prop}

\begin{proof}
Observe that $U(\circ\cup \diamond)V=((U\circ V)\oplus (U\diamond V))/J$ 
is a pushout for the pair 
$(\hat{\circ},\hat{\diamond})$ in the category of abelian groups.
Furthermore, if $\star\colon U\times V\to Z$ is a bimap and 
$\circ, \diamond\to \star$, then the associated homomorphisms $U\circ V\to Z$ 
and $U\diamond V\to Z$ create a commutative square from $U\otimes V$ to $Z$ (see \figref{fig:join}).  
The universal property of the pushout implies that
$\circ\cup \diamond\to \star$, and we see that $\cup$ is indeed a join.
\figref{fig:join} also illustrates that $\circ\cup \diamond\leftrightarrow \bullet(\ker \hat{\circ}+\ker \hat{\diamond})$.  In similar fashion, $\circ\cap \diamond\leftrightarrow \bullet(\ker \hat{\circ}\cap \ker\hat{\diamond})$ and serves as the meet.
\end{proof}

\begin{remark}\label{rem:meet}
The meet $\circ \cap \diamond$ of bimaps $\circ:U\times V\to Y$ and $\diamond:U\times V\to Y$ 
has appeared elsewhere as a bimap $U\times V\to X\oplus Y$, for example in~\cite{BW:find-isom}*{p. 1976}.  
We have modified our definition of $\circ\cap\diamond$ here to ensure that it is a full bimap.
\end{remark}

\subsection{A Galois connection}
\label{subsec:Galois}
A {\em Galois connection} between partially ordered sets $(P,\leq)$ and $(Q,\subseteq)$ is a pair
of functions $\bot:P\to Q$ and $\top:Q\to P$ such that 
\begin{align*}
	(\forall&  p\in P, \forall q\in Q) & q^{\top} \leq p \quad &\Longleftrightarrow\quad q \subseteq p^{\bot}.
\end{align*}
The reader is referred to \cite{GALOIS} for equivalent definitions and interpretations of Galois connections.
In this section we exhibit a Galois connection between the lattice of factor-equivalence classes
of bimaps on $U\times V$ and the lattice of subsets of the ring $\End(U)\times\End (V)^{{\rm op}}$.  

Recall that we regard $U$ as a right $\End(U)$-module and a left $\End(U)^{\rm op}$-module.
Thus, subsets $S$ of $\End(U)\times\End(V)^{{\rm op}}$ act on the right of $U$ and on the left of $V$ so that one
may form the tensor product $\otimes_S:U\times V\to U\otimes_S V$.  We say that a bimap 
$\circ\colon U\times V\to W$ is \emph{mid $S$-linear} if it factors through $\otimes_S$:
\begin{align}
	(\forall & s\in S,\forall u\in U,\forall v\in V) & us\circ v & = u\circ sv.
\end{align}
For a fixed bimap $\circ:U\times V\to W$, recall from \eqref{def:adj} that
\begin{align*}
	\Adj(\circ) & = \{ (x,y)\in \End(U)\times \End(V)^{{\rm op}} : \forall u\in U,\forall v\in V,
		ux\circ v = u\circ yv \},
\end{align*}
a subring of $\End(U)\times \End(V)^{\rm op}$. We can now
formulate the Galois connection.

\begin{thm}
\label{thm:Galois}
Let  $U$ and $V$ be abelian groups. If
$S\subseteq \End(U)\times \End(V)^{{\rm op}}$, and 
$\circ\colon U\times V\to W$ is a bimap, then
$$
\otimes_{S} \to \circ~\mbox{if, and only if,}~S \subseteq \Adj(\circ).
$$ 
This establishes a Galois connection between the lattice of factor-equivalence classes of
bimaps on $U\times V$ and the lattice of subsets of $ \End(U) \times \End(V)^{{\rm op}}$. 
Moreover,
\begin{enumerate}[(i)]
\item $\Adj(\otimes_{\Adj(\circ)}) =\Adj(\circ)$, and
\item $\otimes_{\Adj(\otimes_S)}  
\leftrightarrow \otimes_S$.
\end{enumerate}
Hence $\Adj(\otimes_{-})$ and $\otimes_{\Adj(-)}$ are closure operators
on the two lattices.
\end{thm}

\begin{proof}
First, if  $\circ$ factors through $\otimes_S$ then there is a (unique) map 
$\hat{\circ}\colon U\otimes_S V\to W$ such that
$u\circ v=(u\otimes v)^{\hat{\circ}}$ for all $u\in U$ and all $v\in V$. For each $s\in S$,
\begin{align*}
	(\forall & u\in U,\forall v\in V) & 
us\circ v & =(us\otimes v)^{\hat{\circ}}=(u\otimes sv)^{\hat{\circ}}=u\circ sv.
\end{align*}
So $S\subseteq\Adj(\circ)$. 
Conversely, suppose $S\subseteq\Adj(\circ)$. Define
$\tau\colon U\otimes_S V\to U\circ V$ on pure tensors, sending $u\otimes v\mapsto u\circ v$,
and extending linearly. For each $s\in S$,
\begin{align*}
	(\forall & u\in U,\forall v\in V) & 
(us\otimes v)^{\tau} & =us\circ v=u\circ sv=(u\otimes sv)^{\tau}.
\end{align*}
It follows that $\tau$ is well-defined, and hence that $\circ$ factors through $\otimes_S$.

Both (i) and (ii) are general properties of Galois connections, but their proof in context is straight-forward.
For (i), we have $\Adj(\circ)\subseteq\Adj(\otimes_{\Adj(\circ)})$. For the reverse inclusion,
we know that
$\otimes_{\Adj(\circ)}\to \circ$, so there exists
$\tau\colon U\otimes_{\Adj(\circ)}V\to U\circ V$ such that 
$u\circ v=(u\otimes v)^{\tau}$ for all $u\in U$, $v\in V$.
Let $s=(x,y)\in\Adj(\otimes_{\Adj(\circ)})$. 
\begin{align*}
	(\forall & u\in U, \forall v\in V) & 
	us\circ v & =(ux\otimes v)^{\tau}=(u\otimes yv)^{\tau}=u\circ sv,
\end{align*}
so that $s\in \Adj(\circ)$.
Similarly for (ii), we have $\otimes_{\Adj(\otimes_S)}\to \otimes_S$. 
Since $S\subseteq \Adj(\otimes_S)$, the identity map on pure
tensors extends linearly to a well-defined map $U\otimes_S V\to U\otimes_{\Adj(\otimes_S)}V$.
It follows that $\otimes_S\to\otimes_{\Adj(\otimes_S)}$.
\end{proof}

A consequence of \thmref{thm:Galois} is that a 
bimap $\circ$ on $U\times V$ is a tensor product (i.e. it possesses the universal mapping property for some set 
$S\subseteq\End(U)\times\End(V)^{\rm op}$) if, and only if,  $\circ\leftrightarrow \otimes_{\Adj(\circ)}$.  
This proves the last assertion of \thmref{thm:main1}.

Note that if $R$ is any ring with multiplication $\cdot:R\times R\to R$, then
$\Adj(\cdot_R)=\{(u\mapsto ur, v\mapsto rv) : r\in R\}\cong R$.  In that sense adjoint rings
are arbitrary.  Their representations, however, are more constrained, in the sense that a 
subring $S$ of 
$\End(U) \times \End(V)^{{\rm op}}$ seems rather unrelated to its closure,  $\Adj(\otimes_S)$. 
For instance, there are commutative subrings $S$ of $\End(U)\times\End(V)^{\rm op}$, having
nontrivial Jacobson radical, for which $\Adj(\otimes_S)$ is noncommutative and simple.

\section{Autotopisms and Normalisers}
\label{sec:auto}
Having introduced the homotopism category of bimaps and some of its basic properties,
we now consider the automorphism groups in the category. 

In Section~\ref{subsec:aut-norm} we show, for an arbitrary
bimap $\circ\colon U\times V\to W$, that the autotopism group
$\Aut(\circ)$ is naturally represented as a normaliser, $N(\Adj(\circ))$, within $\Aut(U)\times\Aut(V)$, thereby
completing the proof of~\thmref{thm:main1}. In Section~\ref{subsec:norm-struc}, 
we describe $N(A)$ for semiprimary separable subrings $A$ of $\End(U)\times\End(V)^{{\rm op}}$, 
and hence also $\Aut(\otimes_S)$ for $S\subseteq \End(U)\times\End(V)^{{\rm op}}$ having $\Adj(\otimes_S)$
is semiprimary and separable.  This includes the autotopisms of tensor products of 
finite-dimensional vector spaces. Finally,
in Section~\ref{subsec:alg-norm}, we present an algorithm to construct $N(A)$.

\subsection{Autotopisms acting on adjoints}
\label{subsec:aut-norm}
For abelian groups $U,V$, and subring $A$ of
$\End(U)\times \End (V)^{{\rm op}}$, define the \emph{normaliser} of $A$ to be
\begin{equation}
\label{eq:norm A}
\begin{split}
	N(A) & = \left\{(f,g)\colon \forall (x,y)\in A,~(x,y)^{(f,g)}=(f^{-1} x f, gyg^{-1})\in A\right\} \\
	& \subseteq  \Aut(U) \times \Aut(V).
\end{split}
\end{equation}

\begin{thm}
\label{thm:autotope}
Let $U$ and $V$ be abelian groups. 
\begin{enumerate}[(i)]
\item If $\circ$ is a bimap on $U\times V$ then
$\Aut(\circ)|_{\Aut(U)\times \Aut(V)}\subseteq N(\Adj(\circ))$.
\item If $S\subseteq  \End(U)\times \End(V)^{{\rm op}}$ then $\Aut(\otimes_S)|_{\Aut(U)\times \Aut(V)}=N(\Adj(\otimes_S))$.
\end{enumerate}
\end{thm}

\begin{proof}
For (i), let $(f,g;h)\in \Aut(\circ)$.  For each $(x,y)\in \Adj(\circ)$ and all $u\in U,v\in V$,
\begin{align*}
	ux^f\circ v &  = (uf^{-1}x\circ g^{-1} v)^h
		 = (uf^{-1}\circ yg^{-1} v)^h =  u\circ y^g v.
\end{align*}
Therefore, $(x,y)^{(f,g)}\in \Adj(\circ)$, so that $\Aut(\circ)|_{\Aut(U)\times \Aut(V)}$ normalizes $\Adj(\circ)$.

For (ii) we require the reverse containment in the case that $\circ\leftrightarrow \otimes_S=\otimes_{\Adj(\otimes_S)}$ for some
$S\subseteq \End(U)\times \End(V)^{\rm op}$.
Suppose that $(f,g)\in N(\Adj(\otimes_S))$.  We construct
$h\in \Aut(U\otimes_S V)$ such that $(f,g;h)\in \Aut(\otimes_S)$.   If such $h$ exists, it is uniquely 
defined by
$(u\otimes v)^h=uf\otimes gv$, for each $u\in U$ and each $v\in V$.  Accordingly $h$ exists if this definition
is well-defined (respects the tensor product relations).  For all $u,u'\in U$ and all $v,v'\in V$,
\begin{align*}
	((u+u')\otimes v)^h & = uf\otimes gv + u'f\otimes gv = (u\otimes v)^h+(u'\otimes v)^h,\textnormal{ and}\\
	(u\otimes (v+v'))^h & = uf\otimes gv + uf\otimes gv' = (u\otimes v)^h+(u\otimes v')^h.
\end{align*}
Finally, let $s=(x,y)\in S\subseteq \Adj(\otimes_S)$.  As $(f,g)\in N(\Adj(\otimes_S))$, it follows that $(x^f, y^g)\in \Adj(\otimes_S)$.  Now, for each $u\in U$ and each $v\in V$,
\begin{align*}
	(us\otimes v)^h & = ux f\otimes gv = uf x^f \otimes gv = uf \otimes y^g g v
		= uf\otimes g y v = (u\otimes sv)^h.
\end{align*}
Therefore $h$ is well-defined on $U\otimes_S V$, so that $(f,g;h)\in \Aut(\otimes_S)$.  As $\circ \leftrightarrow \otimes_S$, $\Aut(\otimes_S)|_{\Aut(U)\times\Aut(V)}=\Aut(\circ)|_{\Aut(U)\times \Aut(V)}$, which completes the proof.
\end{proof}
\medskip

\noindent {\bf Proof of~\thmref{thm:main1}.}
The first assertion of~\thmref{thm:main1} is
\thmref{thm:autotope}(i). As commented earlier, the second assertion follows from
the Galois connection; specifically, from 
\thmref{thm:Galois}(i).\hfill$\Box$

\subsection{Normalisers of matrix rings}
\label{subsec:norm-struc}
\thmref{thm:autotope} states that  the groups $\Aut(\circ)$
act as automorphisms of the rings $\Adj(\circ)$.  We can use the well-developed structure   
of rings to limit the behavior of $\Aut(\circ)$.  
\medskip

In this section we pursue a more precise description of $N(\Adj(\circ))$ for
$k$-bilinear maps $U\times V\to W$, where $k$ is a field, and $U,V$ and $W$ 
are finite-dimensional $k$-vector spaces.  The proof actually applies to any adjoint ring
which is separable and semiprimary.  Since we are unable at present to describe which subrings
of $\End(U)\times \End(V)^{\rm op}$ are adjoint rings we assume anything is possible.
The resulting structure theorem (\thmref{thm:norm}) is technical, but
each component is implied by well-known properties of rings.   
We need this level of detail for timing estimates of various algorithms,
such as those in Section~\ref{subsec:alg-norm} and \cite{BW:co-rank-2}.
The ring-theoretic properties we use are found in \cite{CR}*{Sections 3,5,6}.
\medskip

Before stating the main structure theorem, we set up some notation and establish some preliminary results.
Fix a field $k$, finite-dimensional $k$-spaces $U,V$, and
let $A$ be a $k$-subalgebra of $\End_k(U)\times \End_k(V)^{\rm op}$.  
Let $J=J(A)$ be the Jacobson radical of $A$.  Define the {\em radical series} of $A$ to be the finite module chain
\begin{align}
\label{eq:rad-series}
	U\oplus V > UJ\oplus JV > \cdots > UJ^c \oplus J^c V > UJ^{c+1} \oplus J^{c+1} V= 0.
\end{align}
As each ideal $J^i$ is characteristic in $A$, 
the radical series is $N(A)$-invariant.  For convenience we let $J^0=A$.  
The main theme is to choose bases for $U$ and $V$ such that
$A$ is block-upper triangular, and such that the action of $N(A)$ on $A$ 
permutes the blocks within each radical section
but otherwise respects the block decomposition.  We choose these bases 
with some additional properties in mind.
\smallskip

First, Wedderburn's principal theorem 
\cite{Jac:basicII}*{p. 374} establishes the existence of 
a subalgebra $S\leq A$ such that $A=J\oplus S$ as a $k$-vector
space.  We call this a {\em Wedderburn decomposition} of $A$. Since $S$ is semisimple, 
$UJ^i$ splits in $U$ as an $S$-module for each $i\in\{0,\ldots,c\}$; likewise $J^i V$ splits in $V$. 
Hence, there are $S$-submodules $X_0,\ldots,X_c\leq U$ and $Y_0,\ldots,Y_c\leq V$ 
such that for all $0\leq i\leq c$,
\begin{align}
\label{eq:define-Xi-Yi}
	 UJ^i & = X_i\oplus \cdots \oplus X_c & 
	 J^i V & = Y_i\oplus \cdots \oplus Y_c.
\end{align}
For each $i\in \{0,\dots,c\}$, $J$ is in the kernel of the induced action of $A$ on $UJ^i/UJ^{i+1}$ and
$J^iV/J^{i+1}V$, so these modules are $(A/J)$-modules, and hence semisimple.  Since $S$ splits 
with $J$ in $A$,  $UJ^i/UJ^{i+1}\cong X_i$ as $S$-modules. Therefore, bases for $U$ and $V$ that exhibit
the decompositions $U=X_0\oplus\cdots \oplus X_c$ and $V=Y_0\oplus \cdots\oplus Y_c$ will express $A$ in 
block-upper triangular form with $S$ represented as block-diagonal matrices. 
\medskip

It is convenient to state our main structure theorem with reference to the group
\begin{align}
\label{eq:define-NSJ}
	N(S;J) & = \Stab_{N(S)}(\{(UJ^i,J^iV): 0\leq i\leq c\}).
\end{align}
Our preliminary results describe some structural properties of this group.
Let $\mathcal{E}$ be the set of central-primitive
idempotents of $S$ \cite{CR}*{Section 3}. Then the minimal ideals of $S$
are precisely the ideals $I=eSe=Se$, where $e\in\mathcal{E}$. Next, we
define an equivalence relation $\sim$ on $\mathcal{E}$, where
\begin{align}
e\sim e' &~~  \Longleftrightarrow ~~ \begin{array}{l} eSe\cong e'Se'~\mbox{as rings, and for all}~i\in\{0,\dots, c\}, \\
\dim X_i e=\dim X_i e'~\mbox{and}~\dim eY_i=\dim e'Y_i. \end{array}
\end{align}
With this notation, we have the following result.

\begin{lemma}
\label{lem:perms}
Let $S$ be a semisimple complement to the Jacobson radical, $J$, of a subalgebra
of $\End_k(U)\times\End_k(V)^{{\rm op}}$, and let $\mathcal{E}$ be the set of central-primitive idempotents of $S$. 
Then, for each $e\in \mathcal{E}$, the following hold.
\begin{enumerate}
\item[(i)] For $i\in \{0,\ldots,c\}$, $X_ie$ is a direct sum of isomorphic simple $S$-submodules of $X_i$,
and $eY_i$ is a direct sum of isomorphic simple $S$-submodules of $Y_i$.
\item[(ii)] For $e'\in \mathcal{E}$, $e\sim e'$ if, and only if, there exists $g\in N(S;J)$ with $e^g=e'$ such that 
$g$ acts as the identity on $\mathcal{E}-\{e,e'\}$.
\end{enumerate}
\end{lemma}
\begin{proof}
For each $e\in \mathcal{E}$, $X_i e$ is a  faithful $eSe$-module and, as $e$ is central-primitive, $eSe$ is simple.
Therefore, each $X_i e$ is a direct sum of isomorphic simple $eSe$-modules and  every $S$-submodule
of $X_i$ that is isomorphic to a submodule of $X_i e$ is contained in $X_i e$; see \cite{CR}*{Section 3}. 

Now suppose $e'\in\mathcal{E}-\{e\}$.  If $e\sim e'$, then $eSe\cong e'Se'$ and both are simple 
Artian rings. As such, each is isomorphic to $M_{d_e}(\Delta_e)$ for some positive integer $d_e$ 
and finite-dimensional division algebra $\Delta_e$ over $k$.  
Thus, both $X_i e$ and $X_i e'$ are direct sums of multiple copies of $\Delta_e$.   
Since $\dim X_i e=\dim X_i e'$ for each $i$, it follows that $X_i e\cong X_i e'$ as $\Delta_e$-vector spaces.  
Hence, there is a $\Delta_e$-semilinear isomorphism from $X_ie$ to $X_i e'$.  
The same applies to the right modules $eY_i$ and $e'Y_i$.  

Fix $\Delta_e$-semilinear 
transformations $\varphi_i(e):X_i e\to X_i e'$ and $\psi_i:eY_i\to e'Y_i$.  
As $U=\bigoplus_{i=0}^c \bigoplus_{e\in\mathcal{E}} X_i e$  
define $\varphi\in \End(U)$ as $\varphi_i$ from $X_i e\to X_i e'$,
$\varphi_i^{-1}:X_i e'\to X_i e$, and as the identity on $X_i f\to X_i f$ for each $f\in\mathcal{E}-\{e,e'\}$.  
Mimic this construction to  create $\psi\in \End(V)^{\rm op}$ which interchanges $eY_i$ with $e'Y_i$ 
for each $i\in \{0,\dots,c\}$.  It follows that $e^{(\varphi,\psi)}=e'$, and
$(\varphi,\psi)\in N(S;J)$.
\end{proof}

As in the proof of \lemref{lem:perms},
for each $e\in\mathcal{E}$, $eSe\cong \M_{d_e}(\Delta_e)$ for a positive integer $d_e$ and finite-dimensional
division $k$-algebra $\Delta_e$. Furthermore, for each $i\in \{0,\dots,c\}$ there is are pairs $(m_i(e), n_i(e))$ of 
non-negative integers such that
\begin{align}\label{eq:x-y}
	X_i e & \cong  \Delta_e^{d_e}\otimes_k k^{m_i(e)} 
		& \textnormal{ and  }\qquad eY_i & \cong \Delta_e^{d_e}\otimes_k k^{n_i(e)}
\end{align}
as $eSe$-modules. The following is an explicit description of the subgroup of $N(S;J)$
that normalizes every simple ideal of $S$.

\begin{lemma}
\label{lem:final-stab}
Let $S$ be a semisimple complement to the Jacobson radical, $J$, of a subalgebra
of $\End_k(U)\times\End_k(V)^{{\rm op}}$, and $\mathcal{E}$ the set of central-primitive idempotents of $S$. 
Then the subgroup
of $N(S;J)$ that normalizes every ideal $eSe$ ($e\in\mathcal{E}$) of $S$ is isomorphic to
\begin{align*}
	\prod_{e\in\mathcal{E}} \Gamma {\rm L}_{d_e}(\Delta_e) \otimes_k 
	\left(
		\prod_{i=0}^c \GL_{m_i(e)}(k) \times \GL_{n_i(e)}(k)
		\right).
\end{align*}
\end{lemma}

\begin{proof}
Let $(\varphi,\psi)\in N(S;J)$ be such that $eSe^{(\varphi,\psi)}=eSe$ for every $e\in\mathcal{E}$.  
It follows that, for each $e\in\mathcal{E}$, $(\varphi,\psi)$ induces a ring automorphism
$\tau_e$ of $eSe\cong \M_{d_e}(\Delta_e)$.  The Skolem-Noether theorem \cite{CR}*{(3.62)} 
shows that
$\tau_e$ is conjugation by  a $\Delta$-semilinear transformation.  
Also, ${\rm \Gamma L}_{d_e}(\Delta_e)$ acts diagonally on 
$U e$, isomorphic as $\Delta_e$-module to 
$\Delta_e^{d_e (m_0+\cdots +m_c)}$, and also on $eV$, isomorphic to $\Delta_e^{d_e (n_0+\cdots + n_c)}$.

Let $\tau=(\tau_e\colon e\in\mathcal{E})\in \End(U)\times \End(V)^{\rm op}$.  Notice $eSe^{(\varphi,\psi)}=eSe^{\tau}$ and 
$(X_i\tau, \tau Y_i)=(X_i, Y_i)$ so $\tau\in N(S;J)$.  Finally, 
$\tau'=(\varphi,\psi)\tau^{-1}$ centralizes $S$ and lies in $N(S;J)$.  Therefore, 
$\tau'$ is the identity on the $S$-simple submodules  of the $S$-semisimple modules $X_i$ and $Y_i$.  In particular, 
$\tau'$ acts on $X_i e$ as $1\otimes_k \GL_{m_i(e)}(k)$ and on $eY_i$ as $1\otimes_k \GL_{n_i(e)}(k)$, 
in the decomposition of \eqref{eq:x-y}.
\end{proof}

We can now state the full structure theorem for $N(A)$.

\begin{thm}
\label{thm:norm}
Let  $A$ be a subalgebra of $\End(U)\times \End(V)^{\rm op}$.  
Let $J=J(A)$ be the Jacobson radical of $A$, and $S$ a semisimple complement to $J$ in $A$. 
Let $\mathcal{E}$ be the set of central-primitive idempotents of $S$, and $N(S;J)$ the group defined
in (\ref{eq:define-NSJ}). Then the following hold.
\begin{enumerate}[(i)]
\item $N(A)=\langle 1+J, N(S)\cap N(A)\rangle$.

\item $N(S)\cap N(A)=\Stab_{N(S)}(J)=\{ (x,y)\in N(S) : J^{(x,y)}=J\}$.

\item For each 
$e\in\mathcal{E}$ there is a positive integer $d_e$ and a finite-dimensional division 
$k$-algebra $\Delta_e$ such that $eSe\cong \M_{d_e}(\Delta_e)$ and
\begin{align*}
	\prod_{e\in\mathcal{E}} \GL_{d_e}(\Delta_e) & \leq N(S)\cap N(A)  \leq N(S;J).
\end{align*}

\item Let $\mathcal{F}=\{\sum_{e'\sim e}e'\colon e\in\mathcal{E}\}$. 
Then $N(S;J)=\prod_{f\in\mathcal{F}}N(fSf;J)$, where
\begin{align*}
	N(fSf;J) & = \Stab_{N(fSf)}(\{(UJ^if,fJ^iV): 0\leq i\leq c\}).
\end{align*}

\item Let $f\in\mathcal{F}$, and suppose $f=\sum_{e'\sim e}e'$ for some $e\in\mathcal{E}$.
Put $d_f=d_e$, $\Delta_f=\Delta_e$, $r_f=|\{e'\in\mathcal{E}\colon e'\sim e\}|$,
$d_f m_i(f)=\dim_{\Delta_e} UJ^i e/UJ^{i+1}e$, and
$d_f n_i(f)=\dim_{\Delta_e} eJ^iV/eJ^{i+1}V$.  Then
\begin{align*}
	N(fSf;J) & = \left( {\rm \Gamma L}_{d_f}(\Delta_f) \otimes_k 
		\prod_{j=0}^c \GL_{m_{i}(f)}(k)\times \GL_{n_{i}(f)}(k)\right)\wr S_{r_f}.
\end{align*}
\end{enumerate}
\end{thm}

\begin{proof}
For (i), we recall the theorem of Mal'cev asserting that $1+J\leq A^{\times}\leq N(A)$ acts transitively on 
the Wedderburn decompositions of $A$. Thus, for each $\varphi\in N(A)$ there exists $z\in J$ with 
$S^{\varphi}=S^{1+z}$, so that $(1+z)\varphi^{-1}\in N(S)\cap N(A)$.

For (ii),  since $N(S)\cap N(A)$ acts as ring automorphisms on $A$, it follows that $N(S)\cap N(A)\leq \Stab_{N(S)}(J)$.
Also, if $(\varphi,\psi)\in \Stab_{N(S)}(J)$ and $(x,y)\in A$, then $(x,y)=(w+s,z+t)$ for $w,z\in J$ and $s,t\in S$.  
As $(w,z)^{(\varphi,\psi)}\in J$ and $(s,t)^{(\varphi,\psi)}\in S$, we have 
$(x,y)^{(\varphi,\psi)}=(w,z)^{(\varphi,\psi)}+(s,t)^{(\varphi,\psi)}\in J+S=A$,  
so that $(\varphi,\psi)\in N(A)$.

For (iii), we have $S=\bigoplus_{e\in\mathcal{E}} eSe$ with $eSe\cong \M_{d_e}(\Delta_e)$. Hence, 
$\prod_{e\in\mathcal{E}} \GL_{d_e}(\Delta_e)=S^{\times}\leq N(S)\cap N(A)$.  
Clearly, $N(S)\cap N(A)$ stabilizes the radical series (\ref{eq:rad-series}), so
$N(S)\cap N(A)\leq N(S;J)$.

Finally, (iv) and (v) follow directly from Lemmas~\ref{lem:perms} and \ref{lem:final-stab}.
\end{proof}

\subsection{An algorithm to construct \boldmath$N(A)$}
\label{subsec:alg-norm}
In this section, we present an algorithmic version of \thmref{thm:norm}. Although we anticipate practical
uses for such an algorithm as a stand-alone function, it is already proving to be a valuable component
in the algorithmic study of $p$-groups. We discuss this matter further in the concluding section.
The complexity of the algorithm is difficult to predict, but it
is roughly a function of the size of the Jacobson radical of $\Adj(\circ)$.  For convenience, we shall
mostly think of $k$ in this section as a finite field, although extensions to algebraic number fields are possible.
\medskip

Let $A\leq\M_a(k)\times\M_b(k)$ be given (we assume, as the enveloping algebra of some set of generators).
First, we compute a Wedderburn decomposition $A=J\oplus S$, along with a decomposition of $S$ into its minimal ideals
$\{eSe:e\in\mathcal{E}\}$, where $\mathcal{E}$ is the set of central-primitive idempotents of $S$.  R\'{o}nyai has shown that 
the complexity of decomposing algebras in this way is essentially that of factoring
polynomials over $k$ \cite{Ronyai}.  If $k$ is a finite field, this can be done in polynomial time using randomized
algorithms of the {\em Las Vegas} variety. 
(Such algorithms only return answers that are correct, but there is also a small chance that failure is reported.) 
Details may be found in
\citelist{\cite{CI:algebra}*{Section 4}\cite{EG:semisimple}\cite{Iv:algebra}}.  

Next, the idempotents $\mathcal{E}$ are partitioned to form the $\mathcal{F}$ of 
\thmref{thm:norm}(iv).
This applies the equivalence relation $\sim$ of \lemref{lem:perms}.
In particular, if $e,e'\in\mathcal{E}$ then $e\sim e'$ requires only that $d_e=d_{e'}$, that 
$\Delta_e\cong \Delta_{e'}$, and that dimensions of various $\Delta_e$-vector spaces agree.  
Of those requirements, 
the only significant challenge is to test whether $\Delta_e\cong \Delta_{e'}$. 

In fact, to build the permutations
in $N(S;A)$ promised by \lemref{lem:perms}, we really need an explicit isomorphism between the two
division algebras.  
In the case when $k$ is a finite field, isomorphism type is determined by dimension, and
we require an isomorphism between field
extensions $K$ and $L$ of $k$. In the context of the algorithm, the extensions $K$ and $L$ are 
specified, respectively, 
by generators $\rho$ and $\mu$ (matrices of the same degree with entries in $k$).
We use the following idea suggested to us by W.M. Kantor:
compute the minimal polynomial of $\rho$ over the base field $k$, and factor this polynomial
over $L$; then, for any root $\tau\in L$, the assignment $\rho\mapsto \tau$ determines a linear transformation
conjugating $K$ to $L$.  
\smallskip

We remark that isomorphism testing of general division algebras over $\mathbb{Q}$ is not known to be easy
 except for quaternionic instances \cite{IR:quat}.
\smallskip

The final step, for each $e\in\mathcal{E}$ and $0\leq i\leq c$, is to decompose 
$X_ie =\Delta_e^{d_e}\otimes_k k^{m_i(e)}$ 
and $eY_i=\Delta_e^{d_e}\otimes_k k^{n_i(e)}$.  This is done with a randomized Las Vegas algorithm 
known as the MeatAxe \citelist{\cite{HR:meataxe}\cite{IL:meataxe}}.  The action of 
${\rm \Gamma L}_{d_e}(\Delta_e)\otimes (\GL_{m_i(e)}(k)\oplus \GL_{n_i(e)}(k))$ on $(X_ie, eY_i)$ 
is then immediate from the decomposition.  We have thus proved:

\begin{thm}
For finite fields $k$,
there is a polynomial-time Las Vegas algorithm that given a $k$-subalgebra $A\leq \M_u(k)\times\M_v(k)$
computes generators for the group $\langle 1+J, N(S;J)\rangle$, along with its order and composition factors.
\end{thm}

Finally, by \thmref{thm:norm}(ii), to build $N(A)=\langle 1+J, \Stab_{N(S;J)}(J)\rangle$ one must next construct 
$\Stab_{N(S;J)}(J)$.  In general, it seems that one can do little better than simply to 
build a permutation presentation of $N(S;J)$ on $J$ and compute the stabilizer as a permutation group, 
of course taking advantage  of the decomposition in \thmref{thm:norm}(iii).  
The problem of finding stabilizers in permutation groups is 
thought to be difficult  \cite{Luks}*{Section 4}, and we do not expect an efficient general solution to the problem
(see, for example, the construction in Section~\ref{subsec:example}). 
We have, however, established the following result.

\begin{thm}
\label{thm:norm-alg}
There is a polynomial-time Las Vegas algorithm that, given a semisimple subalgebra 
of $\End(U)\times \End(V)^{\rm op}$, where $U$ and $V$ finite-dimensional vector spaces over a finite
field, constructs generators for $N(A)$.
\end{thm}

\subsection{An example}
\label{subsec:example}
We conclude this section with a construction which shows that 
computing $N(A)$ is at least as hard as computing $\Aut(\circ)$ 
for an arbitrary bimap $\circ$, and that the latter is essentially a generic
``quadratic stabilizer" problem for which no efficient solution is known (details in Section~\ref{sec:quad-stab}).
Thus, it is likely not through a lack of understanding that we have failed
to achieve polynomial time for the general problem.  
We stress, however, that not all rings are adjoint rings, and in fact the rings we construct are not known
to be adjoint rings. Hence, although the examples in our family give some
indication of the difficulty of constructing $\Aut(\otimes_S)$ for $S\subset\End(U)\times \End(V)^{\rm op}$,
they do not completely settle the matter.
\smallskip

Fix a field $k$, any bimap $\circ\colon U\times V\to W$, where $U,V$ and $W$ are finite-dimensional
$k$-spaces, and ordered bases $\mathcal{X}$ and $\mathcal{Y}$ for $U$ and $V$ respectively. 

For each $\varphi\in W^*=\hom_k(W,k)$, let
$M(\circ^{\varphi})$ denote the {\em Gram matrix} of the $k$-bilinear
form $\circ^{\varphi}$, whose $(x,y)$-entry ($x\in \mathcal{X}$, $y\in\mathcal{Y}$) is
$(x\circ y)\phi\in k$.  The {\em Gram representation} of $\circ$ is then defined as
\begin{align}
	W^{\circ} & = \{ M(\circ^{\varphi}) : \varphi\in W^*\}\leq \M_{a\times b}(k),
\end{align}
where $a=\dim_k U$ and $b=\dim_k V$.  Observe that $(f,g;h)\in \Aut(\circ)$ if, and only if, the matrices
$(F,G)$ corresponding to $(f,g)$ satisfy the condition
\begin{align}
	FW^{\circ} G & = W^{\circ}.
\end{align}
Now define 
\begin{align*}
	A & = 
\left\{~\left(\begin{bmatrix} a1_U & Z \\ 0 & b1_V \end{bmatrix},
	\begin{bmatrix} a1_U & 0\\ Z^t & b1_V \end{bmatrix} \right) 
	\colon a,b\in k,\;Z\in W^{\circ}\right\}\leq \M_{a+b}(k)\times\M_{a+b}(k).
\end{align*}
Then
$J=J(A)=\left\{ \left(\left[ \begin{smallmatrix} 0 & Z \\ 0 & 0 \end{smallmatrix} \right],
\left[\begin{smallmatrix} 0 & 0 \\ Z^t & 0 \end{smallmatrix}\right]\right) \colon Z\in W^{\circ}\right\}$, so
\begin{align*}
	N(S;J) & = 
		\left\{ \left(\begin{bmatrix} F & 0 \\ 0 & G \end{bmatrix},
				\begin{bmatrix} F^{-t} & 0 \\ 0 & G^{-t} \end{bmatrix}\right)
	 \colon F\in\GL(a,k), 
			G\in\GL(b,k)\right\}, \textnormal{ and }\\
	N(S)\cap N(A) & = \left\{ \left(\begin{bmatrix} F & 0 \\ 0 & G^{-t} \end{bmatrix},
				\begin{bmatrix} F^{-t} & 0 \\ 0 & G\end{bmatrix}\right) \colon
		FW^{\circ} G = W^{\circ}\right\}\cong \Aut(\circ).
\end{align*}  
Thus, in order to construct $N(S)\cap N(A)$ from $N(S;J)$, one must solve a generic
stabilizer problem of the form $FW^{\circ} G = W^{\circ}$.

\section{pseudo-isometries and $*$-normalisers}
\label{sec:pseudo}
In this section we consider bimaps that possess a certain form of symmetry. We say that
$\circ\colon V\times V\to W$ is {\em Hermitian} if there exists $\theta\in\GL(W)$ such
that, for all $u,v\in V$, $u\circ v=(v\circ u)^{\theta}$. Such bimaps, which include the
more familiar reflexive forms, were studied
in~\cite{BW:find-isom}, where the groups
\begin{equation}
\label{eq:isom}
\begin{split}
	\Isom(\circ) & =\left\{ f\in \Aut(V)\colon~\forall u,v\in V,~
		uf\circ  vf= (u\circ v)\right\}\\
		& = \{ f \colon (f,g;h)\in \Aut (\circ),~ f=g~\mbox{and}~h=1\}
\end{split}
\end{equation}
of {\em isometries} of $\circ$ were described, and then used to construct intersections of classical groups. 
However, there are crucial applications of 
Hermitian bimaps -- notably to automorphism groups of $p$-groups (see Section~\ref{sec:apps}) -- 
that involve a broader (but still restricted) type of autotopism,
called a {\em pseudo-isometry}. We therefore study the group of all pseudo-isometries,
namely
\begin{equation}
\label{eq:pseudo}
\begin{split}
	\pseudo(\circ) & =\left\{ (f;\hat{f})\in \Aut(V)\times \Aut(W)\colon~\forall u,v\in V,~
		uf\circ  vf= (u\circ v)^{\hat{f}}\right\}\\
		& = \{ (f;h) \colon (f,g;h)\in \Aut (\circ),~ f=g\}.
\end{split}
\end{equation}
As in the case of a general bimap, we propose to study Hermitian bimaps by factoring through an 
associated tensor product. In view of that, and of the specific applications we have in mind, 
we restrict our attention to bimaps that
are either {\em symmetric} ($u\circ v=v\circ u$ for all $u,v\in V$), or {\em alternating} 
($v\circ v=0$ for all $v\in V$). The tensor products associated to  symmetric and alternating 
bimaps are equipped with the same symmetry property, 
and we denote them $\wedge^+$ and $\wedge^-$, respectively.
\smallskip

Once again, we form tensors over the adjoint algebra, $\Adj(\circ)$, of the bimap $\circ$.
The symmetric nature of $\circ$ means that $(x,y)\in\Adj(\circ)$ if, and only if,
$(y,x)\in\Adj(\circ)$. If, in addition, $\circ$ is nondegenerate, then $y$ is uniquely determined by $x$.
Hence $x^*:=y$ defines an anti-automorphism of $\Adj(\circ)$ of order at
most 2, giving it the structure of a $*$-ring. 
If $A\subseteq \End(V)$ is a $*$-ring, then the normaliser of $A$ in~\eqref{eq:norm A} becomes
\begin{equation}
	N^*(A ) = \{ g\in \Aut(V) \colon (y^g)^*=(y^*)^g\in S~\mbox{for all}\;y\in A\}.
\end{equation}

\subsection{Proof of \thmref{thm:main2}}
We now prove our analogue of Theorem~\ref{thm:main1} for symmetric and exterior tensor products.

\begin{proof}
If $\circ\colon V\times V\to W$ is nondegenerate symmetric or alternating then $\circ$ factors through
$\wedge^{\pm}_{\Adj(\circ)}$.  To be equivalent to a tensor product, 
$\hat{\circ}\colon V\wedge_{\Adj(\circ)} V\to W$  must be an isomorphism.

Now suppose that $\circ=\wedge_A^{\pm}$ for a $*$-algebra $A\leq \End(V)$.
Let $\varphi\in N^*(\Adj(\wedge_S))$.  Then $(\varphi,\varphi)\in N(\Adj(\wedge_S))$ and so, 
as in \thmref{thm:autotope}, there is an induced linear mapping $\hat{\varphi}\in \GL(V\otimes_S V)$ 
defined by $(u\otimes v)^{\hat{\varphi}}=u\varphi \otimes v\varphi$.  This map also satisfies
$(u\wedge u)^{\hat{\varphi}}=0$, so we can induce $\hat{\varphi}$ on $V\wedge_S V$.
Now $(\varphi,\hat{\varphi})\in \Psi\Isom(\wedge^{\pm}_S)$.  
In this way, $N^*(\Adj(\wedge^{\pm}_S))\subseteq \Psi\Isom(\wedge^{\pm}_S)|_{\Aut(V)}$.
\end{proof}

\subsection{\boldmath$*$-algebra normalisers}
\label{subsec:norm-star-struc}
In \thmref{thm:main2} we demonstrated that pseu\-do-i\-som\-e\-tries of alternating tensor products are essentially
$*$-normalisers of the adjoint ring of the tensor product.    
We now give a structural description of the $*$-normaliser of an algebra
of matrices; in Section~\ref{subsec:alg-norm-star} we describe an algorithm to construct this group.
\smallskip

We adapt the notation set up in Section~\ref{subsec:norm-struc} 
to $*$-algebras.
Let $A\leq\Bbb{M}_d(k)$ be a $*$-algebra, where $k=2k$ is a field (we exclude fields of
characteristic $2$).  By a result of Taft~\cite{Taft:complements}, $A$ possesses a {\em $*$-invariant}
(semisimple) complement, $S$, to its Jacobson radical $J=J(A)$.  

We also require $\mathcal{E}$
to consist of $*$-invariant central-primitive idempotents. This set is obtained from $\mathcal{E}_0$,
the set of central-primitive idempotents of the {\em ring} $A$ (ignoring $*$ temporarily) as follows.
Put $\mathcal{I}_0=\{e\in\mathcal{E}_0\colon e^*=e\}$ and 
$\mathcal{J}_0=\{e+e^*\colon e\in\mathcal{E}_0-\mathcal{I}_0\}$.
Then $\mathcal{E}:=\mathcal{I}_0\cup\mathcal{J}_0$ is the desired set of $*$-invariant central-primitive idempotents.
In particular, $eSe$ is a minimal $*$-ideal, for every $e\in \mathcal{E}$.  

Our initial partition of idempotents is a little more refined than for ordinary rings. 
Each $*$-simple $*$-subring $eSe$, for $e\in\mathcal{E}^*$, has an associated pair, $(d_e,\mathcal{O}_e)$, 
of parameters, where $d_e$ is a positive integer, and $\mathcal{O}_e$ is a $*$-algebra
whose non-trivial $*$-invariant elements are invertible.
Osborn has classified such rings $\mathcal{O}$ and so we refer these as {\em Osborn pseudo-division algebras}~\cite{Osborn}.
To avoid confusion, we denote the involution in $\mathcal{O}$ as $s\mapsto \bar{s}$.  

Define the usual Hermitian $\mathcal{O}$-forms as bimaps 
$\bullet\colon\mathcal{O}^{d}\times \mathcal{O}^{d}\to\mathcal{O}$ where for some $M=\bar{M}^t\in \M_{d}(\mathcal{O})$,
\begin{align}
	(\forall & u,v\in \mathcal{O}^d) & u\bullet v & = u M\bar{v}^t.
\end{align}
As shown in \cite{Wilson:unique-cent}*{Section 4.5}, for every $e\in\mathcal{E}$ there is a unique Osborn
division algebra $\mathcal{O}_e$, a rank $d_e$, and a nonsingular $M=\bar{M}^t\in \M_d(\mathcal{O})$ such that
\begin{align*}
	eSe & \cong \Adj(\bullet:\mathcal{O}_e^{d_e}\times \mathcal{O}_e^{d_e}\to\mathcal{O}_e) 
		\cong \langle \M_d(\mathcal{O}), X\mapsto M\bar{X}^t M^{-1}\rangle\\
	(eSe)^{\#} & = \{ x\in eTe^{\times} \colon x x^*=1\} = \Isom(\bullet:\mathcal{O}_e^{d_e}\times \mathcal{O}_e^{d_e}\to\mathcal{O}_e)\\
		& = \{ g\in \GL_d(\mathcal{O}) \colon gM\bar{g}^t=M\}.
\end{align*}
Define $S$-submodules $X_0,\ldots,X_c$ of $V$ as in~(\ref{eq:define-Xi-Yi}), 
where $VJ^i=X_i\oplus\ldots\oplus X_c$
for each $0\leq i\leq c$. Finally, 
define an equivalence relation $\sim$ on $\mathcal{E}$, where $e\sim e'$ if, and only if,
$eSe$ and $e'Se'$ are isomorphic as $*$-rings (that is, $d_e=d_{e'}$ and $\mathcal{O}_e$ and $\mathcal{O}_{e'}$
are isomorphic Osborn pseudo-division algebras) and, for all $i\in\{0,\ldots,c\}$, $\dim X_ie=\dim X_ie'$.

The following is our $*$-analogue of~\thmref{thm:norm}.

\begin{thm}
\label{thm:norm-star}
Let $A$ be a $*$-subalgebra of $\End(V)$.
Let $J=J(A)$ be the Jacobson radical of $A$, and $S$ a semisimple $*$-invariant complement to $J$ in $A$.
Let $\mathcal{E}$ be the set of $*$-invariant central-primitive idempotents of $S$. 
Then the following hold.
\begin{enumerate}[(i)]
\item $N^*(A)=\langle \{z+\sqrt{1+z^2}: z\in J, z^*=-z\}, N^*(S)\cap N^*(A)\rangle$.

\item $N^*(S)\cap N^*(A)=\Stab_{N^*(S)}(J^-)\cap\Stab_{N^*(S)}(J^+)$.

\item For each 
$e\in\mathcal{E}$ there is a positive integer $d_e$, a finite-dimensional Osborn pseudo-division 
$k$-algebra $\mathcal{O}_e$, and an Hermitian $\mathcal{O}$-form 
$\bullet_e\colon \mathcal{O}_e^{d_e}\times \mathcal{O}_e^{d_e}\to \mathcal{O}_e$  such that 
$eSe\cong \Adj(\bullet_e)$ (as $*$-algebras) and
\begin{align*}
	\prod_{e\in\mathcal{E}} \Isom(\bullet_{e}) & \leq N^*(S)\cap N^*(A)  \leq \Stab_{N^*(S)}(VJ^i).
\end{align*}

\item Let $\mathcal{F}=\{\sum_{e\sim e'}e'\colon e\in\mathcal{E}\}$.
Then $N(S;J)=\prod_{f\in\mathcal{F}}N(fSf;J)$, where
\begin{align*}
  N(fSf;J) & = \Stab_{N^*(fTf)}(\{VJ^i f\colon 0\leq i\leq c\}).
\end{align*}

\item Let $f\in F$, and suppose $f=\sum_{e'\sim e}e'$ for some $e\in \mathcal{E}$. Put $d_f=d_e$,
$\mathcal{O}_f=\mathcal{O}_e$, $r_f=|\{e'\in \mathcal{E}\colon e'\sim e\}$, and 
$d_f m_i(f)=\rank_{\mathcal{O}_e} VJ^i e/VJ^{i+1}$.  
Then
\begin{align*}
	\Stab_{N^*(fTf)}(UJ^if, fJ^i V) & = \left( \Pseudo(\bullet_e)\otimes_k 
		\prod_{j=0}^c \GL_{m_{i}(f)}(k)\right)\wr S_{r_f}.
\end{align*}
\end{enumerate}
\end{thm}

\begin{proof}
For (i), let $\varphi\in N^*(A)$. Since $S^{\varphi}$ is a $*$-invariant complement to $J$ in $A$, and
$U=\{z+\sqrt{1+z^2}\colon z\in J^-\}$ acts transitively on the set of all such complements, 
there exists $u\in U$ such that
$S^{\varphi u}=S$~\cite{BW:find-isom}*{Theorem 1.1}. 
It follows that $\varphi u\in {\rm Stab}_{N^*(S)}(J^-)\cap{\rm Stab}_{N^*(S)}(J^+)$, and the result
follows.  

For (ii), note that $\varphi\in N^*(S)$ lies in $N^*(A)$ if, and only if, $\varphi$ stabilizes $J$
and commutes with the involution on $J$. The condition is equivalent to $\varphi$ stabilizing $J^+$ and $J^-$.
For, if $\varphi$ stabilizes $J^+$ and $J^-$, and $z=z^++z^-$ with $z^{\pm}\in J^{\pm}$, then
\[
(z^{\varphi})^*=((z^++z^-)^{\varphi})^*=((z^+)^{\varphi})^*+((z^-)^{\varphi})^*=(z^+)^{\varphi}-(z^-)^{\varphi}.
\]
On the other hand, if $z\in J^{\epsilon}$, say, with $z^{\varphi}\not\in J^{\epsilon}$, then $(z^*)^{\varphi}=\epsilon z^{\varphi}\neq
(z^{\varphi})^*$.  

For (iii)-(v) the proofs are essentially the same as that of~\thmref{thm:norm} except that 
$N^*(eSe)\cong\Pseudo(\bullet_e)\otimes \GL_{m_i(e)}(k)$, where 
$\bullet_e\colon \mathcal{O}_e^{d_e}\times \mathcal{O}_e^{d_e}\to \mathcal{O}_e$~\cite{Wilson:unique-cent}*{Corollary 4.30}.  
\end{proof}

\subsection{An algorithm to construct \boldmath$N^*(A)$}
\label{subsec:alg-norm-star}
Most of the machinery needed to provide an algorithmic version of~\thmref{thm:norm-star} was
developed in~\cite{BW:find-isom}.

First,
procedures for decomposing $S$ as a direct sum of minimal $*$-ideals, and for identifying
the simple type of these ideals, are given 
in~\cite[Theorem 4.1]{BW:find-isom}. 

The algorithm for~\thmref{thm:norm-star} is almost identical to its counterpart for~\thmref{thm:norm}.
The only essential difference is that, instead of generators for $\GL(d_i,K_i)$, we must choose
suitable generators for $\Pseudo(\bullet_e)$. Those groups are, however, all (conformal) classical groups, and
it is elementary to write down small generating sets for them (see~\cite[Section 5.4]{BW:find-isom}).

For (ii), an algorithmic version of Taft's decomposition is given in~\cite[Proposition 4.3]{BW:find-isom}. 
The unipotent radical $\{z+\sqrt{1+z^2}\colon z\in J^-\}$ is constructed in~\cite[Section 5.2]{BW:find-isom}
using a power series. Finally, the remarks we made about stabilizing the radical in 
Section~\ref{subsec:alg-norm} apply equally in this setting.

\smallskip
We conclude this section with an analogue of~\thmref{thm:norm-alg} for $*$-rings.

\begin{thm}
\label{thm:norm-star-alg}
There is a polynomial-time Las Vegas algorithm that, given a semisimple $*$-subalgebra, $A$, of $\End(V)$,
where $V$ is a finite-dimensional vector space over a finite field of odd characteristic, 
constructs generators for $N^*(A)$.
\end{thm} 

\section{Applications}
\label{sec:apps}
We conclude the paper with a brief discussion of several algorithmic problems of interest
whose solution relies on our ability to compute and understand $\Aut(\circ)$.

\subsection{Automorphisms of $p$-groups}\label{sec:p-groups}
The relationship between nilpotent groups and algebras extends back to the 1930's and has evolved
to handle ever larger families of groups; for a survey see \cite{Warfield:nil}*{Section 5}.
The typical method is to relate commutation $[x,y]=x^{-1} y^{-1} xy$
in a group $G$ to a distributive product.  
Through specific correspondences of Baer, and of Kaloujnine, Lazard, and Mal'cev,
automorphisms of $G$ are seen to induce autotopisms of a distributive product.
We make this precise for the more elementary setting: the
Baer correspondence \cite{Baer:class2}.  Further details are given in \cite{Wilson:unique-cent}*{Section 3}.
\smallskip 

Let $G$ be a group, $G'=\langle [x,y]=x^{-1} y^{-1} xy \colon x,y\in G\rangle$ its {\em commutator subgroup},
and $Z=Z(G)=\{x\in G\colon [x,G]=1\}$ its {\em center}. Suppose that $G'\leq Z$.
If $V=G/Z$ and $W=G'$, with operations written additively,
then $\circ\colon V\times V\to W$, with $xZ\circ yZ:=[x,y]$ for all $x,y\in G$,
is a well-defined bimap. 
Also, since $v\circ v=0$ for all $v\in V$, we see that $\circ$ is {\em alternating}. Hence
$\circ$ factors uniquely through 
$\wedge\colon V\times V\to V\wedge V=V\otimes V/\langle v\otimes v\colon v\in V\rangle$.

Each $\alpha\in\Aut(G)$ restricts to an automorphism $w^{\hat{\varphi}}=w\alpha$ on $W$, 
and induces an automorphism $(xZ)\varphi=x\alpha Z$ on $V$.  
Furthermore, the pair $(\varphi;\hat{\varphi})$ is a pseudo-isometry of $\circ$.
This establishes a homomorphism from $\Aut(G)$ to the group $\Pseudo(\circ)$ of  all pseudo-isometries of $\circ$.
In some important settings -- for example when $G^p=1$ for some prime $p$ -- the image of $\Aut(G)$
is all of $\Pseudo(\circ)$ \cite{Wilson:unique-cent}*{Proposition 3.8}.

In the absence of more refined strategies, $\Pseudo(\circ)$ is typically  
constructed ``by brute-force", meaning that one simply computes the stabilizer of $\ker\hat{\circ}$ 
under the natural action of $\GL(V)$
on $V\wedge V$, sending $u\wedge v\mapsto ug\wedge vg$, for $g\in \GL(V)$.
The limitations are obvious: 
the action of $\GL(V)$ on $V\wedge V$ can have orbits that are  
far too large for effective computation. Moreover, the results give no hint of structure. 

One way to finesse the problem is to factor $\circ$ through the possibly smaller space 
$V\wedge_{A} V$, where $A=\Adj(\circ)$.  This helps in two ways. First, the natural group 
that acts on $V\wedge_{A} V$, namely the group $\Pseudo(\wedge_A)$ of pseudo-isometries 
of the bimap $\wedge_A$, is no longer all of $\GL(V)$, and may be a significantly smaller subgroup. 
Second, the space $V\wedge_AV$ may have
much smaller dimension than $V\wedge V$.  Therefore finding $\Pseudo(\circ)$ as a stabilizer in
$\Pseudo(\wedge_A)$ of $\ker\hat{\circ}\leq V\otimes_A V$ will often be substantially easier.
Not surprisingly, this approach to computing $\Aut(G)\cong\Pseudo(\circ)$ 
is most effective in situations where $A=\Adj(\circ)$ is large or $V\wedge_A V$ is small.
Both of those desirable conditions are met, for instance, when $|G'|=p^2$,
a particularly nice case that is handled separately in \cite{BW:co-rank-2}. 
\smallskip

The general method we have outlined above constitutes one part of a comprehensive new
strategy to construct generators for the automorphism group of $p$-group of class
2 and exponent $p$. This strategy is currently being developed jointly by the authors 
and E.A. O'Brien~\cite{BOW:auto-pgrp}.

\subsection{Quadratic stabilizer}\label{sec:quad-stab}
Autotopism groups provide a natural context for the general problem of stabilizing a subspace of
rectangular matrices.

The familiar {\em linear stabilizer problem} starts with a field $k$, a positive integer $a$, and subspace $W\leq k^a$; and asks for $\Stab(W)=\{x\in \GL(a,k): Ux=U\}$.  By simply writing $k^a=X\oplus U$ we find that
\begin{align*}
	\Stab(W) & = \left\{\begin{bmatrix} A & B\\ 0 & C\end{bmatrix} : A\in \GL(X), B\in \Hom(X,U),C\in \GL(U)\right\}.\end{align*}
So this linear stabilizer problem is elementary to solve.

The {\em quadratic stabilizer problem} concerns a field $k$, positive integers $a,b$, and a subspace $W\leq \M_{a\times b}(k)$.  The goal is to describe the group 
\begin{align}
	\Stab(W) & = \{(x,y)\in \GL(a,k)\times \GL(b,k)~:~xWy^t=W\}.
\end{align}
The related {\em Hermitian stabilizer problem} has the tighter constraints that $a=b$ and that for all $w\in W$, $w=\varepsilon~\bar{w}^t$ for some $\varepsilon\in \{\pm 1\}$, and some (possibly identity) field automorphism $s\mapsto \bar{s}$ on $k$.  
The problem is then to describe the group
\begin{align}
	H\Stab(W) & = \{x\in \GL(a,k)~:~xW\bar{x}^t=W\}.
\end{align}
The quadratic and Hermitian stabilizer problems are known hard problems.  It is no surprise that
the reverse construction to Section~\ref{subsec:example} shows that the quadratic stabilizer problem is
the problem of constructing $\Aut(\circ)$. 

The introduction of tensor products (other than with $k$) is
new to the this topic.  Similar to the improvements made for automorphisms of $p$-groups in 
Section~\ref{sec:p-groups}, knowledge of
$\Aut(\otimes_S)$ reduces the work needed to compute $\Stab(W)$.

\bibliographystyle{amsplain}

\end{document}